\documentclass[11pt]{amsart}

\usepackage[T1]{fontenc}
\usepackage[utf8]{inputenc}
\usepackage{lmodern}
\usepackage{amsmath,amssymb,amsthm,mathtools}
\usepackage[margin=1in]{geometry}
\usepackage{tikz-cd}

\usepackage[colorlinks=true,linkcolor=blue,citecolor=blue,urlcolor=blue]{hyperref}

\newtheorem{theorem}{Theorem}[section]
\newtheorem{proposition}[theorem]{Proposition}
\newtheorem{corollary}[theorem]{Corollary}
\newtheorem{lemma}[theorem]{Lemma}
\theoremstyle{definition}
\newtheorem{definition}[theorem]{Definition}
\newtheorem{example}[theorem]{Example}
\theoremstyle{remark}
\newtheorem{remark}[theorem]{Remark}

\newcommand{\F}{\mathbb{F}}
\newcommand{\Z}{\mathbb{Z}}
\newcommand{\Index}{\operatorname{Index}}
\newcommand{\id}{\operatorname{id}}
\newcommand{\cI}{\mathcal{I}}
\newcommand{\cS}{\mathcal{S}}

\title{A Cohomological Obstruction to the Existence of $\phi$-Equivariant Maps}
\author[T. F. V. Paiva]{Thales Fernando Vilamaior Paiva}
\address{Universidade Federal de Mato Grosso do Sul, C\^ampus de Aquidauana \\ CEP 79200-000, Aquidauana - MS, Brazil }
\email{thales.paiva@ufms.br}

\date{}

\begin{document}



\begin{abstract}
Let $\phi\colon G\to H$ be a continuous homomorphism between compact Lie
groups, let $X$ be a $G$-space, and let $Y$ be an $H$-space.  We formulate a
filtered change-of-groups obstruction to the existence of a
$\phi$-equivariant map $X\to Y$. We apply them
to the inclusion of a finite cyclic group into the circle and to
homomorphisms between tori acting coordinatewise on products of
odd-dimensional spheres.
\end{abstract}

\maketitle

\section{Introduction}

The existence of equivariant maps is a classical problem in transformation
group theory, and its simplest instances already include results of
Borsuk--Ulam type \cite{Borsuk1933,Dold1983}; general references on compact
transformation groups and equivariant homotopy theory include
\cite{Bredon1972,tomDieck1987,May1996}.  If a compact Lie group $G$ acts on
spaces $X$ and $Y$, cohomological index theories obstruct the existence of a
$G$-equivariant map $X\to Y$ through functorial properties of the associated
Borel constructions.  The ideal-valued index of Fadell and Husseini
\cite{FadellHusseini1988}, its page-wise partial versions
\cite[Section~6.1]{BlagojevicLuckZiegler2015}, and the numerical index of
Volovikov \cite{Volovikov2000} have been used in several forms of this
problem; see, for example, \cite{CoelhoMattosSantos2012}.

In this paper, we consider compact Lie groups $G$ and $H$, a continuous
homomorphism $\phi\colon G\to H$, a $G$-space $X$, and an $H$-space $Y$.  A
map $f\colon X\to Y$ is said to be $\phi$-equivariant when
\[
f(gx)=\phi(g)f(x),\,\,\text{ for every }\,\,g\in G
\text{ and }x\in X,\]
where the conventional notion of a $G$-equivariant map is obtained by taking $G=H$ and
$\phi=\id_G$, whereas the general definition allows the acting group to
change through subgroup inclusions, quotient homomorphisms, or other
continuous homomorphisms.  In this setting, the induced map
\[
(B\phi)^*:H^*(BH;R)\to H^*(BG;R)
\]
must be incorporated into any
cohomological comparison between $X$ and $Y$.

The motivation for this approach comes from the extension problem studied in
\cite{Paiva2025}.  Indeed, if $G\subseteq H$ and an action of $H$ on $X$
extends a given action of $G$, then the identity map of $X$ is
$\iota$-equivariant for the inclusion $\iota\colon G\hookrightarrow H$.
Remark~3.8 of \cite{Paiva2025} suggests that a class killed by a differential
in one Borel spectral sequence, whose restriction remains nonzero in the
other, should obstruct such an extension.  The present paper develops this
observation for arbitrary $\phi$-equivariant maps and may therefore be viewed
as a continuation of the question raised there.

For a fixed acting group, the page-wise kernels of the Borel spectral
sequence are commonly called partial Fadell--Husseini indices.  Their
behavior under group automorphisms and $\phi$-equivariant self-equivalences
was considered in \cite[Lemma~7.2]{BlagojevicLuckZiegler2015}.  Here we
retain the notation $\cI_r^K(Z;R)$ for these ideals and establish their
change-of-groups monotonicity for arbitrary continuous homomorphisms and
possibly different source and target spaces.  Our main result (Theorem \ref{thm:filtered-monotonicity}) states that
the existence of a $\phi$-equivariant map $X\to Y$ implies
\[
 (B\phi)^*\bigl(\cI_r^H(Y;R)\bigr)\subseteq \cI_r^G(X;R)
 \quad\text{for }2\le r\le\infty.
\]
At the limiting page this inclusion gives the corresponding relation between
the Fadell--Husseini indices, while the finite pages retain the differential
information that motivated the construction.  After establishing this
obstruction, we apply it to the inclusion
$C_p\cong\Z/p\Z\hookrightarrow S^1$ and to
homomorphisms between tori acting on products of odd-dimensional spheres; the
latter calculation gives a change-of-groups extension of the
torus example in \cite{CoelhoMattosSantos2012}.

\section{The conventional equivariant-map problem}
\label{sec:conventional}

Throughout the paper, all groups are compact Lie groups, all actions are
continuous, and $R$ is a commutative ring with unit.  We work with spaces for
which the Borel fibrations admit natural multiplicative Leray--Serre spectral
sequences.  Whenever the limiting page is used, strong convergence to Borel
cohomology is also assumed.  These hypotheses hold, in particular, for the
finite equivariant CW complexes occurring in the examples below and in the
standard finitistic setting considered in \cite{Paiva2025}.  Local
coefficient systems will be displayed whenever they are relevant.

Let $K$ act on a space $Z$, choose a contractible free $K$-space $EK$, and
write 
\[Z_K=(EK\times Z)/K=EK\times_KZ
\]
for the Borel construction and
$BK=EK/K$ for the classifying space relative to $K$.  The projection
$\pi_Z\colon Z_K\to BK$ determines the Borel fibration
$Z\to Z_K\to BK$, whose cohomological Leray--Serre spectral sequence will be
denoted by $E_r(Z_K;R)$.  Following the conventions of \cite{McCleary2001},
its second page is
\[
 E_2^{p,q}(Z_K;R)
 \cong H^p\bigl(BK;\mathcal H^q(Z;R)\bigr),
\]
where the local coefficient system is induced by the action of $K$ on
$H^q(Z;R)$.  When $Z$ is path connected, the bottom row is canonically
identified with $E_2^{p,0}(Z_K;R)\cong H^p(BK;R)$.

Suppose now that $X$ and $Y$ are $G$-spaces and $f\colon X\to Y$ is a
$G$-equivariant map.  The rule $f_G[e,x]=[e,f(x)]$ defines a map
$f_G\colon X_G\to Y_G$ over the identity of $BG$ and consequently induces a
contravariant morphism from $E_r(Y_G;R)$ to $E_r(X_G;R)$ whose restriction to
the bottom row on the second page is the identity of $H^*(BG;R)$.  

Recall that the Fadell--Husseini index \cite{FadellHusseini1988} is the
homogeneous ideal
\[
\Index_G(Z;R)
=\ker\bigl(H^*(BG;R)\rightarrow H_G^*(Z;R)\bigr),
\]
hence the morphism at the limit
gives the relation
\[
 \Index_G(Y;R)\subseteq\Index_G(X;R).
\]

This conventional case contains the central idea of the present paper.  The
existence of an equivariant map prevents a cohomology class of $BG$ from
disappearing earlier in the target spectral sequence than in the source
spectral sequence.  For a general homomorphism $\phi\colon G\to H$, the same
principle holds after pulling the class back along $B\phi$.

\section{\texorpdfstring{$\phi$-equivariant maps}{phi-equivariant maps} and Borel constructions}
\label{sec:phi}

\begin{definition}
Let $\phi\colon G\to H$ be a continuous homomorphism, let $X$ be a $G$-space,
and let $Y$ be an $H$-space.  A continuous map $f\colon X\to Y$ is
\emph{$\phi$-equivariant} if $f(gx)=\phi(g)f(x)$ for every $g\in G$ and
$x\in X$.
\end{definition}

Choose compatible models of the universal spaces and a
$\phi$-equivariant map $E\phi\colon EG\to EH$; its quotient is a model for
$B\phi\colon BG\to BH$.  If $f\colon X\to Y$ is $\phi$-equivariant, then
$f_\phi\colon X_G\to Y_H$, defined by
$f_\phi[e,x]=[E\phi(e),f(x)]$, is well defined and gives the following map
of Borel fibrations:
\[
\begin{tikzcd}
X \arrow[r] \arrow[d,"f"'] & X_G \arrow[r,"\pi_X"]
  \arrow[d,"f_\phi"] & BG \arrow[d,"B\phi"] \\
Y \arrow[r] & Y_H \arrow[r,"\pi_Y"'] & BH.
\end{tikzcd}
\]
The induced homotopy class of $f_\phi$ is independent of these choices; this
is the usual change-of-groups map in Borel cohomology.

In accordance with the notation of \cite{Paiva2025}, let $D_r^{p,q}$ denote
the spectral sequence associated with $Y\to Y_H\to BH$ and let $E_r^{p,q}$
denote the spectral sequence associated with $X\to X_G\to BG$.  Their second
pages and convergence are described by
\begin{align*}
 D_2^{p,q}
 &=H^p\bigl(BH;\mathcal H^q(Y;R)\bigr)
   \Rightarrow H_H^{p+q}(Y;R),\\
 E_2^{p,q}
 &=H^p\bigl(BG;\mathcal H^q(X;R)\bigr)
   \Rightarrow H_G^{p+q}(X;R).
\end{align*}
Naturality produces a morphism
$\psi_r^{p,q}\colon D_r^{p,q}\to E_r^{p,q}$ commuting with the differentials,
and on the second page this morphism is induced by $B\phi$ together with the
morphism of local coefficient systems determined by $f^*$.  In particular,
if $X$ and $Y$ are path connected, then
\[
\psi_2^{p,0}=(B\phi)^*\colon H^p(BH;R)\to H^p(BG;R).
\]
The map
\[
 h=h_{(Y_H,X_G;R)}=f_\phi^*\colon
 H_H^*(Y;R)\rightarrow H_G^*(X;R)
\]
preserves the Serre filtrations, and $\psi_\infty$ is the induced morphism
on the associated graded groups.  Thus
$\operatorname{gr}(h)=\psi_\infty$; in general, $\psi_\infty$ should not be
identified with the unfiltered homomorphism $h$ without resolving the
extension data.  This is the cohomological comparison considered in
\cite{Paiva2025}.

\section{A filtered cohomological obstruction}
\label{sec:obstruction}

Because the Borel spectral sequence is first-quadrant, no differential leaves
its bottom row, and hence there are canonical quotient maps
$E_r^{*,0}(Z_K;R)\to E_s^{*,0}(Z_K;R),$ whenever
$2\le r\le s\le\infty$.  In particular, the identification
$E_2^{*,0}(Z_K;R)=H^*(BK;R)$ determines a homomorphism
\[
\rho_r^{K,Z}\colon H^*(BK;R)\to E_r^{*,0}(Z_K;R).
\]

\begin{definition}
Let $K$ be a compact Lie group and $Z$ be a path-connected $K$-space $Z.$ For each $2\le r\le\infty$, define the
\emph{$r$-th partial Borel index} by
\[
 \cI_r^K(Z;R)
 =\ker\bigl(\rho_r^{K,Z}\colon H^*(BK;R)
       \rightarrow E_r^{*,0}(Z_K;R)\bigr).
\]
\end{definition}

For finite $r$, this is the partial Fadell--Husseini index used in
\cite[Section~6.1]{BlagojevicLuckZiegler2015}; the notation $\cI_r^K(Z;R)$
emphasizes the filtration page while avoiding a second index superscript.
Under the multiplicative hypotheses above, $\cI_r^K(Z;R)$ is a homogeneous
ideal, with
\[
 0=\cI_2^K(Z;R)\subseteq\cI_3^K(Z;R)\subseteq\cdots
 \subseteq\cI_\infty^K(Z;R).
\]
Thus $\cI_r^K(Z;R)$ records precisely the
classes in the bottom row that have been hit by incoming differentials before
the $r$-th page.

\begin{theorem}
\label{thm:filtered-monotonicity}
Let $\phi\colon G\to H$ be a continuous homomorphism, let $X$ be a
path-connected $G$-space, and let $Y$ be a path-connected $H$-space.  If there
exists a $\phi$-equivariant map $f\colon X\to Y$, then, for every
$2\le r\le\infty$,
\[
 (B\phi)^*\bigl(\cI_r^H(Y;R)\bigr)
 \subseteq \cI_r^G(X;R).
\]
\end{theorem}

\begin{proof}
Let $\ell\in\cI_r^H(Y;R)$.  Thus the image of $\ell$ in
$D_r^{*,0}=E_r^{*,0}(Y_H;R)$ is zero.  The morphism of spectral sequences
induced by $f_\phi$ carries this image to the image of $(B\phi)^*\ell$ in
$E_r^{*,0}(X_G;R)$.  The latter image is therefore zero, which means that
$(B\phi)^*\ell\in\cI_r^G(X;R)$.
\end{proof}

When $G=H$, $\phi$ is an automorphism, and $f$ is a
$\phi$-equivariant self-homotopy equivalence, applying the theorem also to a
$\phi^{-1}$-equivariant homotopy inverse recovers the equality in
\cite[Lemma~7.2]{BlagojevicLuckZiegler2015}.  In the present result, neither
the acting groups nor the spaces are required to coincide.

Motivated by Theorem~\ref{thm:filtered-monotonicity}, we introduce the
following obstruction set.

\begin{definition}
For $2\le r\le\infty$, set
\[
 \cS_{\phi,R}^{(r)}(Y_H,X_G)
 =\left\{\ell\in\cI_r^H(Y;R)\ \,;\, \ 
 (B\phi)^*\ell\notin\cI_r^G(X;R)\right\}.
\]
The elements of this set will be called \emph{$r$-page witnesses} for the
nonexistence of a $\phi$-equivariant map from $X$ to $Y$.
\end{definition}

\begin{corollary}[Obstruction criterion]
\label{cor:witness}
If $\cS_{\phi,R}^{(r)}(Y_H,X_G)\ne\emptyset$ for some
$2\le r\le\infty$, then there is no $\phi$-equivariant map $X\to Y$.
\end{corollary}

\begin{remark}[The conventional case]
If $G=H$ and $\phi=\id_G$, then Theorem~\ref{thm:filtered-monotonicity}
reduces to $\cI_r^G(Y;R)\subseteq\cI_r^G(X;R)$, so the partial index refines
the usual same-group monotonicity condition by retaining the page on which
each base class disappears.
\end{remark}

\begin{remark}[The limiting ideal]
Under the strong-convergence hypothesis, the bottom edge homomorphism of the
Borel spectral sequence is
$\pi_Z^*\colon H^*(BK;R)\to H_K^*(Z;R)$.  In total degree $p$, the Serre
filtration satisfies $F^{p+1}H_K^p(Z;R)=0$, and hence
$E_\infty^{p,0}\cong F^pH_K^p(Z;R)$.  It follows that the kernel of the edge
homomorphism is exactly the set of bottom-row classes that vanish at
$E_\infty$ \cite{McCleary2001}.
Consequently,
$\cI_\infty^K(Z;R)=\Index_K(Z;R)$.  The limiting instance of
Theorem~\ref{thm:filtered-monotonicity} is therefore
\[
 (B\phi)^*\bigl(\Index_H(Y;R)\bigr)
 \subseteq\Index_G(X;R),
\]
the natural change-of-groups form of Fadell--Husseini monotonicity.
\end{remark}

\begin{proposition}
\label{prop:differential}
Let $\ell\in H^n(BH;R)$.  Suppose that its representative
$\rho_s^{H,Y}(\ell)\in D_s^{n,0}$ is nonzero and that there exists
$\eta\in D_s^{n-s,s-1}$ such that
$d_s(\eta)=\rho_s^{H,Y}(\ell)$.  If the class
$\rho_{s+1}^{G,X}((B\phi)^*\ell)$ is nonzero in $E_{s+1}^{n,0}$, then no
$\phi$-equivariant map $X\to Y$ exists.
\end{proposition}

\begin{proof}
The first hypothesis says that the class represented by $\ell$ belongs to
$\cI_{s+1}^H(Y;R)$, while the second says that its pullback does not belong to
$\cI_{s+1}^G(X;R)$.  Hence $\ell$ determines an element of
$\cS_{\phi,R}^{(s+1)}(Y_H,X_G)$, and the conclusion follows from
Corollary~\ref{cor:witness}.
\end{proof}

\begin{remark}[Relation with the motivating question]
The situation considered in Remark~3.8 of \cite{Paiva2025} requires, in
effect, a class that is killed in one spectral sequence while its pullback
survives in the other.  Proposition~\ref{prop:differential} shows that survival
to $E_\infty$ is stronger than necessary.  In fact, after a class is killed
by $d_s$ in the target sequence, nonvanishing of its pullback on $E_{s+1}$
already contradicts naturality.  The partial indices isolate precisely this
finite-page phenomenon.
\end{remark}

\section{Examples}
\label{sec:examples}

\subsection{The inclusion \texorpdfstring{$C_p\hookrightarrow S^1$}{Cp into S1}}
\label{subsec:cyclic-circle}

We give a model calculation showing that the filtered obstruction can be
sharp.  Let $p$ be an odd prime and let
\[
 C_p=\{z\in S^1\,;\,z^p=1\}
   =\langle e^{2\pi i/p}\rangle
\]
be the subgroup of $p$-th roots of unity in $S^1$.  The map
$[k]\mapsto e^{2\pi i k/p}$ identifies $C_p$ with the cyclic group
$\Z/p\Z$.  Let $\phi\colon C_p\hookrightarrow S^1$ be the inclusion and let us
consider $X=S^{2n-1}\subset\mathbb C^n$ with the scalar action of $C_p,$ and
$Y=S^{2m-1}\subset\mathbb C^m$ with the scalar action of $S^1$.  With
coefficients in $\F_p$, it is well know (cf. \cite{Hatcher2002}) that the identification
$BS^1\simeq\mathbb CP^\infty$ gives 
\[
 H^*(BS^1;\F_p)=\F_p[c],
\] 
where $c$ is a homogeneous element of $\deg c=2$.  On the other hand, the
standard computation of the cohomology of a cyclic group gives
\cite{AdemMilgram2004}
\[
 H^*(BC_p;\F_p)=\Lambda(a)\otimes\F_p[b],
\] 
where $\deg a=1$, $\deg b=2$ and $b$ is the Bockstein of $a$.  Restriction of
the universal complex line bundle, together with naturality of the first
Chern class, gives $(B\phi)^*(c)=b$
\cite{AdemMilgram2004,MilnorStasheff1974}.

The Borel fibration associated with the scalar $S^1$-action on
$S^{2m-1}$ is the unit-sphere bundle of the $m$-fold sum of the universal
complex line bundle.  In the Leray--Serre spectral sequence of an oriented
sphere bundle, the transgression of the fiber fundamental class is the Euler
class of the corresponding vector bundle
\cite{McCleary2001,MilnorStasheff1974}.  Thus, if $u$ generates
$H^{2m-1}(S^{2m-1};\F_p)$, then $d_{2m}(u)=c^m$; similarly, if $v$ generates
$H^{2n-1}(S^{2n-1};\F_p)$, then the source spectral sequence satisfies
$d_{2n}(v)=b^n$.

\begin{proposition}
\label{prop:spheres}
For the actions above, there exists a $\phi$-equivariant map
$S^{2n-1}\to S^{2m-1}$ if and only if $n\le m$.
\end{proposition}

\begin{proof}
Suppose first that $m<n$.  The class $c^m$ is killed by $d_{2m}$ in the
spectral sequence for $Y_{S^1}$, and hence
$c^m\in\cI_{2m+1}^{S^1}(Y;\F_p)$.  Its pullback is $b^m$, which remains
nonzero on the $(2m+1)$-st page because the first differential that can kill
a nonzero polynomial class in the source bottom row is $d_{2n}$ and
$2m<2n$.  Therefore
$b^m\notin\cI_{2m+1}^{C_p}(X;\F_p)$, so $c^m$ belongs to
$\cS_{\phi,\F_p}^{(2m+1)}(Y_{S^1},X_{C_p})$ and obstructs the existence of the
map.

If $n\le m$, the coordinate inclusion obtained by adjoining $m-n$ zero
coordinates maps $S^{2n-1}$ into $S^{2m-1}$ and satisfies
$f(\zeta z)=\phi(\zeta)f(z)$ for every $\zeta\in C_p$, which proves that it is
$\phi$-equivariant.
\end{proof}

This example is deliberately elementary, but it exhibits the two main
features of the method: the obstruction is relative to the homomorphism
$\phi$, and the relevant information is the page on which the Euler class is
killed.

\begin{remark}
The same existence criterion holds for $p=2$.  In that case
$H^*(BC_2;\F_2)=\F_2[t]$ with $\deg t=1$, the restriction of $c$ is $t^2$,
and the source transgression is $d_{2n}(v)=t^{2n}$.  The preceding argument
then applies without further changes.
\end{remark}

\subsection{Homomorphisms between tori and products of spheres}
\label{subsec:tori}

We begin by recalling the conventional equivariant-map problem that motivates
the present example.  In  \cite[Remark 5.3]{CoelhoMattosSantos2012}, the authors consider a example where $G=T^2=S^1\times S^1,$  $X=S^5\times S^5$ and $Y=S^3\times S^3,$ endowed with free $T^2$-actions, and justifying that there is no
$T^2$-equivariant map from $X$ to $Y$.  Their argument is formulated in
terms of the numerical index of a $G$-space.  Indeed,
$H^q(X;\Z)=0$ for $0<q<5$, which gives the corresponding lower bound for
the index of $X$, whereas $H^6(Y/T^2;\Z)=0$ because
$\dim(Y/T^2)=4$.  Since
$BT^2=\mathbb{C}P^\infty\times\mathbb{C}P^\infty$ has nontrivial
cohomology in degree six, the nonexistence conclusion follows from the
general criterion established in that paper.

Our purpose here is to recover this conclusion for the standard
coordinatewise actions and, more generally, to obtain a change-of-groups
and ideal-valued extension of the corresponding calculation.  Thus, we
allow the acting groups on the domain and codomain to be different tori,
and we replace ordinary equivariance by equivariance with respect to an
arbitrary continuous homomorphism between them.

For tuples of positive integers
$\mathbf n=(n_1,\ldots,n_r)$ and
$\mathbf m=(m_1,\ldots,m_s)$, we put
\[
 X_{\mathbf n}=\prod_{i=1}^r S^{2n_i-1},\]
\[
 Y_{\mathbf m}=\prod_{j=1}^s S^{2m_j-1}.
\]
Let $T^r=(S^1)^r$ act on $X_{\mathbf n}$ coordinatewise, with the
$i$-th circle acting by scalar multiplication on
$S^{2n_i-1}\subset\mathbb C^{n_i}$.  Analogously, let $T^s$ act
coordinatewise on $Y_{\mathbf m}$.  Both actions are free.

Every continuous homomorphism $\phi\colon T^r\to T^s$ is represented by
an integer matrix $A=(a_{ji})\in M_{s\times r}(\Z)$ through
\[
 \phi(z_1,\ldots,z_r)_j
   =\prod_{i=1}^r z_i^{a_{ji}},
\]
for $1\le j\le s.$ Writing $ H^*(BT^r;\Z)=\Z[x_1,\ldots,x_r]$ and  $H^*(BT^s;\Z)=\Z[y_1,\ldots,y_s],$ where all the generators have degree two, the homomorphism induced by
$B\phi$ is determined by
\[
 (B\phi)^*(y_j)=L_j
 :=\sum_{i=1}^r a_{ji}x_i,
\]
for $1\le j\le s.$ Since the coordinatewise actions are free, their Borel constructions are
homotopy equivalent to the corresponding orbit spaces, therefore
\[
 (X_{\mathbf n})_{T^r}
 \simeq X_{\mathbf n}/T^r
 \cong\prod_{i=1}^r\mathbb{C}P^{n_i-1},
\]
and an analogous description holds for $Y_{\mathbf m}$.  Consequently, we have $\Index_{T^r}(X_{\mathbf n};\Z)
   =(x_1^{n_1},\ldots,x_r^{n_r})$ and 
   $\Index_{T^s}(Y_{\mathbf m};\Z)
   =(y_1^{m_1},\ldots,y_s^{m_s}).$
In this setting, the finite-page information can also be computed
explicitly.

\begin{lemma}
\label{lem:torus-partial-indices}
For every integer $q\ge2$, one has
\[
 \cI_q^{T^r}(X_{\mathbf n};\Z)
 =
 \bigl(x_i^{n_i}\,;\, 2n_i<q\bigr).
\]
Analogously,
\[
 \cI_q^{T^s}(Y_{\mathbf m};\Z)
 =
 \bigl(y_j^{m_j}\,;\, 2m_j<q\bigr).
\]
\end{lemma}

\begin{proof}
The Borel spectral sequence associated with $X_{\mathbf n}$ is the tensor
product of the sphere-bundle spectral sequences corresponding to its
factors.  Its second page is
\[
 E_2^{*,*}
 =
 \Z[x_1,\ldots,x_r]\otimes
 \Lambda(u_1,\ldots,u_r),
\]
where $u_i$ is the fundamental cohomology class of
$S^{2n_i-1}$ and has degree $2n_i-1$.  After choosing orientations, the
transgressions associated with the sphere factors are $d_{2n_i}(u_i)=x_i^{n_i}$;
see \cite{McCleary2001,MilnorStasheff1974}.  The multiplicative structure
and the derivation property of the differentials show that the relations
acquired in the bottom row before the $q$-th page are precisely those
generated by the classes $x_i^{n_i}$ satisfying $2n_i<q$.
\end{proof}

The limiting obstruction of
Theorem~\ref{thm:filtered-monotonicity} now has the following explicit
form.

\begin{proposition}
\label{prop:torus-obstruction}
If there exists a $\phi$-equivariant map $f : X_{\mathbf n}\rightarrow Y_{\mathbf m},$
then 
\[\bigl(L_1^{m_1},\ldots,L_s^{m_s}\bigr)
 \subseteq
 (x_1^{n_1},\ldots,x_r^{n_r})\]
 in $\Z[x_1,\ldots,x_r],$ where $L_j=\sum_{i=1}^r a_{ji}x_i$.  Consequently, no such map exists if
\[
 \left(\sum_{i=1}^r a_{ji}x_i\right)^{m_j}
 \notin
 (x_1^{n_1},\ldots,x_r^{n_r}),
\]
for some $j\in\{1,\ldots,s\}$.
\end{proposition}

\begin{proof}
At the limiting page,
Theorem~\ref{thm:filtered-monotonicity} gives
\[
 (B\phi)^*
 \bigl(\Index_{T^s}(Y_{\mathbf m};\Z)\bigr)
 \subseteq
 \Index_{T^r}(X_{\mathbf n};\Z).
\]
Substituting the previously computed index ideals and using
$(B\phi)^*(y_j)=L_j$ yields
\[
 (B\phi)^*
 (y_1^{m_1},\ldots,y_s^{m_s})
 =
 (L_1^{m_1},\ldots,L_s^{m_s})
 \subseteq
 (x_1^{n_1},\ldots,x_r^{n_r}),
\]
as required.
\end{proof}

The page-wise obstruction gives a particularly simple numerical
consequence.

\begin{corollary}
\label{cor:torus-simple}
If $a_{ji}\ne0$ and $m_j<n_i$ for some pair $(j,i)$, then there is no
$\phi$-equivariant map
$X_{\mathbf n}\to Y_{\mathbf m}$.
\end{corollary}

\begin{proof}
By Lemma~\ref{lem:torus-partial-indices}, we have $ y_j^{m_j}
 \in
 \cI_{2m_j+1}^{T^s}(Y_{\mathbf m};\Z),$ while its pullback is
\[
 (B\phi)^*(y_j^{m_j})
 =
 L_j^{m_j}
 =
 \left(\sum_{k=1}^r a_{jk}x_k\right)^{m_j}.
\]
Since $a_{ji}\ne0$, this polynomial contains the pure monomial
$a_{ji}^{m_j}x_i^{m_j}$ with nonzero coefficient.  On the other hand,
\[
 \cI_{2m_j+1}^{T^r}(X_{\mathbf n};\Z)
 =
 \bigl(x_k^{n_k}\,;\, n_k\le m_j\bigr).
\]
Because $m_j<n_i$, the monomial $x_i^{m_j}$ is not divisible by any
generator of this monomial ideal.  It follows that
\[
 L_j^{m_j}
 \notin
 \cI_{2m_j+1}^{T^r}(X_{\mathbf n};\Z).
\]
Therefore $y_j^{m_j}$ determines an element of
\[
 \cS_{\phi,\Z}^{(2m_j+1)}
 \bigl((Y_{\mathbf m})_{T^s},
       (X_{\mathbf n})_{T^r}\bigr),
\]
and the conclusion follows from Corollary~\ref{cor:witness}.
\end{proof}

We next explain precisely how the conventional torus example is recovered.

\begin{example}
\label{ex:torus-conventional}
Taking $r=s=2$, $\phi=\id_{T^2}$,
$\mathbf n=(3,3)$, and $\mathbf m=(2,2)$,  then $X_{\mathbf n}=S^5\times S^5$ and $Y_{\mathbf m}=S^3\times S^3.$ Since $\phi$ is the identity, one has $L_1=x_1$ and $L_2=x_2$.
Proposition~\ref{prop:torus-obstruction} would require
\[
 (x_1^2,x_2^2)\subseteq(x_1^3,x_2^3)
\]
if a $T^2$-equivariant map existed.  This inclusion is false because,
for example, $x_1^2\notin(x_1^3,x_2^3).$ Hence there is no $T^2$-equivariant map $S^5\times S^5\rightarrow S^3\times S^3,$ for the standard coordinatewise actions.

The filtered obstruction describes more precisely where this failure
occurs, since by Lemma~\ref{lem:torus-partial-indices},
\[
 \cI_5^{T^2}(S^3\times S^3;\Z)
   =(y_1^2,y_2^2)\,\,
 \text{ and }\,\,
 \cI_5^{T^2}(S^5\times S^5;\Z)=0.
\]
Thus $y_1^2$, for instance, has already disappeared from the bottom row
of the target spectral sequence, while its pullback
$(B\phi)^*(y_1^2)=x_1^2$ still survives in the source spectral sequence.
Equivalently, the relevant transgressions are $d_4(v_j)=y_j^2$
 for $S^3\times S^3$ and $d_6(u_i)=x_i^3$
for $S^5\times S^5.$ Consequently, $y_1^2$ is a fifth-page witness, and
Corollary~\ref{cor:witness} obstructs the equivariant map before the
limiting page is reached.

This recovers, for the coordinatewise actions, the nonexistence conclusion
of the torus example in
\cite{CoelhoMattosSantos2012}.  The present calculation refines that
conclusion by identifying both a specific obstructing class and the first
page on which the incompatibility is visible.
\end{example}

The following example illustrates the genuine change-of-groups situation.

\begin{example}\label{ex:torus-change-of-groups}
Let $r=2$, $s=1$, $\mathbf n=(3,3)$, and $\mathbf m=(2)$, and consider the homomorphism $\phi: T^2\rightarrow S^1,$ give by  $\phi(z_1,z_2)=z_1z_2.$ The induced homomorphism in cohomology satisfies $ (B\phi)^*(y)=x_1+x_2.$ 

If there existed a $\phi$-equivariant map $S^5\times S^5\rightarrow S^3,$ then Proposition~\ref{prop:torus-obstruction} would imply $(x_1+x_2)^2\in(x_1^3,x_2^3).$ But this is impossible, since
\[
 (x_1+x_2)^2=x_1^2+2x_1x_2+x_2^2
\]
does not belong to $(x_1^3,x_2^3)$.

Notice that the acting groups have different ranks in this example.
Although the $S^1$-action on $S^3$ is free, its pullback along $\phi$ is
a $T^2$-action with nontrivial kernel and is therefore not free.  Thus this
example is not covered directly by a result formulated only for free
$G$-spaces acted upon by the same group.
\end{example}

\end{document}